\documentclass[preprint,sort&compress,final,12pt]{elsarticle}
\usepackage[T1]{fontenc}
\usepackage{lmodern}
\usepackage{microtype}
\usepackage{amsthm}
\usepackage{amsmath}
\usepackage{amssymb}
\usepackage{natbib}
\usepackage[hidelinks]{hyperref}
\usepackage{geometry}
\usepackage{bm}
\usepackage{mathrsfs}

\newtheorem{theorem}{Theorem}[section]

\newtheorem{definition}[theorem]{Definition}
\newtheorem{lemma}[theorem]{Lemma}
\newtheorem{proposition}[theorem]{Proposition}

\numberwithin{equation}{section}

\journal{\empty}
\date{}

\begin{document}
\begin{frontmatter}
\title{Propagation dynamics of an acid-mediated invasion model with degenerate tumor diffusion}

\author[au1]{Xinyue Cao}
\ead{xinyue.cao@mail.bnu.edu.cn}
\address[au1]{School of Mathematical Sciences, Beijing Normal University and Laboratory of Mathematics and Complex Systems, Ministry of Education, Beijing 100875, China}

\author[au2]{Quentin Griette\corref{cor1}}
\ead{quentin.griette@univ-lehavre.fr}
\address[au2]{Universit\'{e} Le Havre Normandie, Normandie Univ., LMAH UR 3821, 76600 Le Havre, France}

\author[au1]{Xiong Li}
\ead{xli@bnu.edu.cn}

\author[au3,au4]{Yang Wang}
\ead{ywang2005@sxu.edu.cn}
\address[au3]{School of Mathematics and Statistics, Shanxi University, Taiyuan, Shanxi 030006, China}
\address[au4]{Key Laboratory of Complex Systems and Data Science of Ministry of Education, Shanxi University, Taiyuan, Shanxi 030006, China}

\cortext[cor1]{Corresponding author.}

\begin{abstract}
We investigate traveling wave fronts for an acid-mediated tumor invasion model with density-dependent degenerate diffusion. The model is a partially diffusive PDE--ODE system of Gatenby--Gawlinski type, in which the tumor diffusion coefficient $D(U)$ is allowed to be a general decreasing function satisfying $D(1)=0$. This degeneracy causes the traveling wave equation for the tumor component to lose uniform ellipticity near the healthy state, and hence standard arguments for nondegenerate reaction diffusion systems are not directly applicable. To overcome this difficulty, we introduce a nonlinear change of variables which removes the degeneracy from the highest-order term of the tumor equation. For each fixed admissible tumor profile, the acid profile is represented by a Green kernel, while the healthy-tissue profile is obtained from an explicit integral formula. The transformed tumor profile is then constructed as the stationary limit of a uniformly parabolic auxiliary problem. By combining comparison principles, local Schauder estimates, carefully chosen super- and sub-solutions, and the Schauder fixed-point theorem, we prove the existence of traveling wave fronts for every wave speed $\theta\ge 2\sqrt{rD(0)}$. The resulting wave connects the tumor-dominant state $(0,1,1)$ at $z=-\infty$ to the healthy state $(1,0,0)$ at $z=+\infty$. We further establish strict pointwise bounds, monotonicity of all wave components, and one-sided exponential asymptotic estimates in both the transformed variable and the original traveling-wave variable.

\vspace{0.2cm}
\end{abstract}
\end{frontmatter}

\newpage
\section{Introduction}
In 1996, to explain the fundamental driving factors of the Warburg effect, Gatenby and Gawlinski \cite{1996} proposed the acid-mediated invasion hypothesis and derived the following set of equations, known as the Gatenby-Gawlinski model
\begin{equation}\label{eq:acid model}
\begin{cases}
u_t=u(1-u)-d_1uw,\\[0.2cm]
v_t=D((1-u)v_x)_x+rv(1-v),\\[0.2cm]
w_t=\Delta w+c(v-w),
\end{cases}
\end{equation}
where $u$, $v$, $w$ represent the density functions of healthy cells, tumor cells and lactic acid in tissue microenvironment respectively. Moreover, $D, d_1, r, c$ are positive non-dimensional parameters: $D$ denotes the diffusion rate of tumor cells, $d_1$ denotes the death rate of healthy cells caused by lactic acid, $r$ denotes the intrinsic growth rate of tumor cells, and $c$ denotes the kinetic rate of lactic acid.

Several variations of the original model \eqref{eq:acid model} can be found in the literature that describe different biological mechanisms. For instance, in \cite{gm22}, the concentration of lactic acid $w$ undergoes diffusion at a constant rate and increases proportionally to the unknown $v$ until it reaches the saturation level $w = v$. To reduce the complexity of model \eqref{eq:acid model}, Gallay and Mascia \cite{gm22} consider the limit case $c\to\infty$, whereby the original dynamical differential equation describing the diffusion and growth of the lactic acid concentration $w$ is directly replaced by the algebraic constitutive relation $w=v$. This simplification yields the reduced model
\begin{equation}\label{eq:simple}
\begin{cases}
u_t=u(1-u)-d_1uv,\\[0.2cm]
v_t=D((1-u)v_x)_x+rv\left(1-v\right).
\end{cases}
\end{equation}

As observed in \cite{Casciari JJ 1992,cellular necrosis}, tumors rarely thrive in environments with a pH more acidic than 6.3, which suggests that excessive acidity may inhibit tumor-cell survival. Building on this insight, McGillen et al. \cite{J.B. McGillen JMB} proposed that incorporating a term for acid-induced tumor cell mortality aligns with biological realism and added the term for acid-induced tumor cell death into \eqref{eq:acid model}. Then \eqref{eq:acid model} is transformed into the McGillen-Gaffney-Martin-Maini model
\begin{equation}\label{eq:generalized acid model}
\begin{cases}
u_t=u(1-u-a_2v)-d_1uw,\\[0.2cm]
v_t=D((1-u)v_x)_x+rv\left(1-a_1u-v\right)-d_2vw,\\[0.2cm]
w_t=w_{xx}+c(v-w),
\end{cases}
\end{equation}
where $d_2$ describes the mortality rate of tumor cells induced by lactic acid. For more general models, Li et al. \cite{2024 Lifang ZAMP} considered
\begin{equation}\label{eq:The most general acid model}
\begin{cases}
u_t=(D_1(v)u_x)_x+u\left(1-u-a_2v\right)-d_1uw,\\[0.2cm]
v_t=(D_2(u)v_x)_x+rv\left(1-a_1u-v\right)-d_2vw,\\[0.2cm]
w_t=\Delta w+c(v-w),
\end{cases}
\end{equation}
where $D_1(\cdot)$ and $D_2(\cdot)$ are non-negative functions representing the diffusion rates of healthy cells and tumor cells, respectively.

To gain deeper insight into tumor invasion dynamics, researchers have extensively investigated the aforementioned models via numerical and qualitative analysis. In the numerical analysis of \eqref{eq:generalized acid model}, McGillen et al. adopted numerical simulations and nonstandard asymptotic analysis within a traveling wave framework in \cite{J.B. McGillen JMB}. Through these approaches, they clarified the diverse tumor behaviors captured by the model and revealed how key parameters regulate the speed and profile of invasive tumor waves, further predicting the conditions for the formation of tumor interstitial space. In terms of qualitative analysis on bounded domains, a series of theoretical results have been established for systems \eqref{eq:acid model}, \eqref{eq:generalized acid model} and \eqref{eq:The most general acid model} under homogeneous Neumann boundary conditions. For the original acid model \eqref{eq:acid model}, Fu and Cui \cite{2009 Global existence and stability} proved the global existence of classical solutions using energy estimates and bootstrap arguments, and established the global asymptotic stability of equilibria via Lyapunov functions. Later, Tao, Qi and Zhou \cite{tqz21} further explored the long-time behavior of solutions by applying the method of super- and sub-solutions. Regarding the generalized acid model \eqref{eq:generalized acid model}, Tao and Tello derived the global stability of heterogeneous steady states using the rectangular method in \cite{Tao}. Subsequently, Li, Yao and Yu improved these results in \cite{2023 Li Fang JDE}, where they established the linear and global stability of all equilibria by combining Lyapunov function techniques, the rectangular method, and arguments based on super- and sub-solutions. Most recently, in \cite{2024 Lifang ZAMP}, the same authors extended their previous work to investigate the linear and global stability of the most generalized model \eqref{eq:The most general acid model}. In particular, they achieved a complete characterization of the linear stability of equilibria by reformulating the linearized stability analysis as an algebraic problem.

One interesting feature of the Gatenby-Gawlinski model \eqref{eq:acid model} is that numerical results confirm the existence of invasive traveling wave fronts, which characterize the invasion of tumor cells into healthy tissues. For the simplified model \eqref{eq:simple}, Gallay and Mascia employed smooth diffeomorphisms and phase-plane analysis to prove the existence of traveling wave solutions. Specifically, for the speed $\theta>0$, they obtained wave solutions connecting the equilibria $(1,0)$ and $(0,1)$, as well as those connecting $(1,0)$ and $(1-d_1,1)$. The former solution corresponds to complete tumor invasion, while the latter describes a coexistence state of tumor cells and healthy tissues. For the original model \eqref{eq:acid model}, under the parameter settings $D=c=r=1$ and $0<d_1<1$, Chen and Qi \cite{ChenQi2025} established the existence of traveling wave solutions connecting $(1,0,0)$ and $(1-d_1,1,1)$ for the speed $\theta>2\sqrt{d_1}$. Moreover, such wave solutions are unique up to translation. The core technique proposed in \cite{ChenQi2025} is to decompose the coupled system into three independent subproblems, analyze the existence, monotonicity and asymptotic behavior of solutions for each equation separately, and transform the original problem into a fixed point problem for a composite mapping. The main difficulty lies in analyzing non-autonomous differential equations with infinite boundary conditions.

From the above statements, inspired by the work in \cite{2024 Lifang ZAMP}, we further investigate an acid-mediated tumor invasion model with more general diffusion functions $D(u)$, which satisfy the following hypothesis
\begin{enumerate}
\item[($H_1$)] {$D(1)=0$, $D(u)>0$ for $u\in[0,1)$, $D\in C^1([0,1])$, and $D'(u)<0$ for $u\in[0,1]$.}
\end{enumerate}
Then \eqref{eq:acid model} is converted into
\begin{equation}\label{eq:main-model}
\begin{cases}
u_t=u(1-u)-d_1uw,\\[0.2cm]
v_t=(D(u)v_x)_x+rv(1-v),\\[0.2cm]
w_t=w_{xx}+c(v-w),
\end{cases}
\end{equation}
where $d_1,r,c>0$. In the present work, we focus on the malignant invasion regime $d_1>1$; its biological interpretation in terms of the relevant equilibria will be discussed in Section 2.

The density-limited diffusion term in the second equation of \eqref{eq:main-model} constitutes the core feature of the proposed model, which is both degenerate and nonlinear. Consequently, classical methods for linear reaction-diffusion systems are no longer applicable to this model. Density-dependent diffusion terms have been widely studied in chemotaxis systems. For the parabolic-elliptic chemotaxis system
\begin{equation}\label{eq:Shen model}
\left\{\begin{array}{ll} u_t={\Delta u}-\chi \nabla \cdot (u\nabla v)+u(1-u),\\[0.2cm]
0={\Delta v}-v+u,
\end{array}\right.
\end{equation}
Salako and Shen \cite{Salako Shen 2017 DCDSB} applied the method of super- and sub-solutions and the Schauder fixed-point theorem to prove the existence of traveling wave solutions connecting $(0,0)$ and $(1,1)$ when $0<\chi<\frac{1}{2}$. When $0<\chi<1$, for every nonnegative initial $u_0(\cdot)$ with non-empty compact support $\operatorname{supp}u_0$, they further obtained the propagation speed of \eqref{eq:Shen model}. Another model is the two-component density-suppressed motility reaction-diffusion system,
\begin{equation}\label{eq:Wang model}
\left\{\begin{array}{ll} u_t=\Delta(\gamma(v)u)+u(a-bu),\\[0.2cm]
v_t=\Delta v+u-v,
\end{array}\right.
\end{equation}
where $\gamma(v)$ is called the motility function satisfying
\begin{equation*}
\gamma(v)=\frac{1}{(1+v)^m},\quad m>0,
\end{equation*}
where $a$, $b>0$ are positive constants accounting for the growth and death rates of bacterial cells. Li and Wang employed the method of super- and sub-solutions and parabolic estimates to prove the existence of traveling wave solutions of \eqref{eq:Wang model} connecting $(0,0)$ and $(\frac{a}{b},\frac{a}{b})$ in \cite{2021 Lijing wangzhian}.

To overcome the nonlinear diffusion term, they first proved that an auxiliary linear diffusion equation admits a solution by constructing appropriate super- and sub-solutions. Then they showed that this solution satisfies an elliptic equation. The traveling wave solution is obtained as a fixed point of the operator corresponding to the elliptic equation.

Another feature of \eqref{eq:main-model} is the coupling of ODE and PDE, which can also be found in \cite{OuChunhua}. Here, they considered Lotka-Volterra competition reaction-diffusion system with partial diffusion
\begin{equation*}
\left\{\begin{array}{ll} u_t=u_{xx}+u(1-u-k_1v),\\[0.2cm]
v_t=v(1-v-k_2u).
\end{array}\right.
\end{equation*}
They treated $u$ as a known function in the second equation and solved for $v$. Then they substituted the resulting expression for $v$ into the first equation to obtain an integro-differential equation on the unknown $u$ alone. The above two key challenges also exist in \eqref{eq:main-model}, and we will overcome them by combining the methods in \cite{OuChunhua,2021 Lijing wangzhian,Salako Shen 2017 DCDSB,ChenQi2025} to establish the existence of traveling wave solutions describing the invasion of cancer cells into healthy tissue.

We now summarize the main difficulties and the contributions of the present paper. The construction is not a direct consequence of the published result of Chen and Qi~\cite{ChenQi2025}, nor a direct application of the auxiliary parabolic fixed point method used by Li and Wang~\cite{2021 Lijing wangzhian}. In the published work of Chen and Qi~\cite{ChenQi2025}, the argument is developed for the special linear degeneracy $D(U)=1-U$ and for the coexistence end state in the corresponding parameter regime; hence it does not directly yield the complete invasion front considered here. The method of Li and Wang~\cite{2021 Lijing wangzhian} provides an auxiliary parabolic fixed point strategy, but it cannot be applied directly to the present model before the degeneracy in the tumor equation is removed. To overcome these difficulties, we first introduce a change of variables to eliminate the degeneracy from the highest-order term, and then construct the transformed tumor profile as the stationary limit of an auxiliary parabolic problem. The use of a general decreasing coefficient $D(U)$ and the complete invasion limiting conditions require additional work to close the fixed point construction. In particular, since the transformed coefficient is non-autonomous and may degenerate at one end, we construct adapted generalized super- and sub-solutions; the comparison functions are constructed on the whole line from positive solutions of suitable linear equations. This combined construction is the main technical point of the proof.

The rest of this paper is organized as follows. In Section~\ref{Preliminaries and Main Results}, we formulate the traveling wave problem, introduce the order interval used in the fixed point argument, and state the main results. In Section~\ref{sec:construction}, for each admissible tumor profile, we construct the corresponding acid profile and healthy-tissue profile, introduce the degenerate-variable transformation, and construct the transformed tumor profile through an auxiliary parabolic problem. In Section~\ref{sec:existence}, we prove the invariance, compactness and continuity of the fixed point map, obtain a fixed point by the Schauder fixed-point theorem, and then establish the strict bounds, monotonicity, end states and one-sided asymptotic estimates of the obtained traveling wave.

\section{Preliminaries and main results}\label{Preliminaries and Main Results}
In this section, we present some necessary preliminaries and state our main results. For the diffusion-free version of \eqref{eq:main-model}, four equilibrium points exist (as in \cite{producing lactic acid,J.B. McGillen JMB,2023 Li Fang JDE,2024 Lifang ZAMP}), which are listed as follows
\begin{enumerate}
\item[$\bullet$] $e_1=(0,0,0)$:\ absence of all tissues and acid, unstable.
\item[$\bullet$] $e_2=(1,0,0)$:\
healthy cells at their carrying capacity and absence of both tumor cells and acid, unstable.
\item[$\bullet$] $e_3=(1-d_1,1,1)$:\ coexistence of tumor cells at their carrying capacity and healthy cells at a diminished level, unstable if $d_1>1$, stable if $0<d_1<1$.
\item[$\bullet$] $e_4=(0,1,1)$:\ tumor cells at their carrying capacity, and no healthy cells; stable if $d_1>1$.
\end{enumerate}

It is demonstrated in \cite{producing lactic acid} that $d_1=1$ acts as the critical threshold governing the transition of tumor behavior from benign coexistence to malignant invasion. Accordingly, the parameter regime $d_1>1$ corresponds to the malignant case, where tumor cells tend to invade and replace healthy tissues. This invasion scenario can be mathematically characterized by traveling wave solutions connecting the tumor-dominant equilibrium $(0,1,1)$ to the healthy equilibrium $(1,0,0)$. This explains the standing assumption $d_1>1$ adopted throughout the paper. Based on the above biological background and equilibrium analysis, we now rigorously define the traveling wave solutions investigated in this work.

\begin{definition}
A nonnegative solution $(u(x,t),v(x,t),w(x,t))$ is called a traveling wave solution of \eqref{eq:main-model} connecting $e_4$ to $e_2$ with wave speed $\theta$ if it is of the form
\begin{equation*}
(u(x,t),v(x,t),w(x,t))=(U(z),V(z),W(z)),\quad z=x-\theta t,
\end{equation*}
and satisfies
\begin{equation}\label{eq:tw-system}
\begin{cases}
\theta U'+U(1-U-d_1W)=0,\\[0.2cm]
(D(U) V')' +\theta V'+rV(1-V)=0,\\[0.2cm]
W''+\theta W'+c(V-W)=0,
\end{cases}
\end{equation}
where the prime $'$ denotes differentiation with respect to the traveling wave variable $z$. In addition, the solution satisfies the following asymptotic boundary conditions at infinity
\begin{equation}\label{eq:tw-bc}
(U(-\infty),V(-\infty),W(-\infty))=(0,1,1),\qquad (U(+\infty),V(+\infty),W(+\infty))=(1,0,0).
\end{equation}
In what follows, any traveling wave solution of \eqref{eq:main-model} refers exclusively to the nonnegative solution of the boundary value problem \eqref{eq:tw-system}-\eqref{eq:tw-bc}.
\end{definition}

We aim to establish the existence of traveling wave solutions for \eqref{eq:main-model} by transforming the original boundary value problem \eqref{eq:tw-system}-\eqref{eq:tw-bc} into a standard fixed point problem. To this end, we first introduce the notations and function spaces required for the subsequent analysis. Consider the following two characteristic equations
\begin{equation}\label{eq:beta-characteristic-pre}
\lambda^2-\theta\lambda+a_*=0,
\end{equation}
and
\begin{equation}\label{eq:W-characteristic}
\lambda^2+\theta\lambda-c=0,
\end{equation}
where $a_*:=rD(0)$. For any wave speed satisfying $\theta\ge 2\sqrt{a_*}$, direct computation shows that \eqref{eq:beta-characteristic-pre} and \eqref{eq:W-characteristic} admit real roots given by

\begin{equation}\label{eq:lambda-pm0}
\lambda_1^-:=\frac{\theta-\sqrt{\theta^2-4a_*}}{2}>0,\qquad\lambda_1^+:=\frac{\theta+\sqrt{\theta^2-4a_*}}{2}>0,
\end{equation}
and
\begin{equation}\label{eq:lambda-pm}
\lambda_2^-:=\frac{-\theta-\sqrt{\theta^2+4c}}{2}<0,\qquad\lambda_2^+:=\frac{-\theta+\sqrt{\theta^2+4c}}{2}>0.
\end{equation}
Clearly, $\lambda_1^-\le\lambda_1^+$, with equality when $\theta=2\sqrt{a_*}$. We further define $\nu_*$ as the positive root of
\begin{equation*}
\nu^2+\theta\nu-a_*=0,
\end{equation*}
namely,
\begin{equation}\label{eq:define-gamma-star}
\nu_*:=\frac{-\theta+\sqrt{\theta^2+4a_*}}{2}>0.
\end{equation}
Based on \eqref{eq:lambda-pm0}--\eqref{eq:define-gamma-star}, we choose
\begin{equation}\label{eq:lambda-choice}
0<\lambda< \min\left\{ \frac{r}{\theta}, -\lambda_2^- \right\}, \qquad 0<\mu< \frac{\nu_*}{D(0)}.
\end{equation}

Throughout the paper, $\lambda_1^\pm$ and $\lambda_2^\pm$ denote the roots of (2.3) and (2.4), respectively, while $\lambda$ and $\mu$ in (2.8) are auxiliary decay parameters.

With the above notations, we introduce the function space of bounded uniformly continuous functions
\begin{equation*}
X:=C^b_{\mathrm{unif}}(\mathbb R)=\left\{v\in C(\mathbb R): v\hbox{ is uniformly continuous on }\mathbb R \hbox{ and }\|v\|_\infty<\infty\right\},
\end{equation*}
where $\|v\|_\infty:=\|v\|_{L^\infty(\mathbb R)}$. Then $X$ is a Banach space equipped with the norm $\|\cdot\|_\infty$.

For any $L>0$, we define
\begin{equation*}
\mathcal K_L:=\left\{v\in X: \underline V_L(z)\le v(z)\le\overline V(z),\ z\in\mathbb R\right\},
\end{equation*}
where
\begin{equation*}
\overline V(z):=\begin{cases} 1, & z\le 0,\\[0.2cm]
e^{-\lambda z}, & z>0,
\end{cases}\qquad
\underline V_L(z):=\begin{cases} 1-e^{\mu(z+L)}, & z<-L,\\[0.2cm]
0, & z\ge -L.
\end{cases}
\end{equation*}
By direct verification, these functions satisfy
\begin{equation}\label{eq:barrier-order-L}
0\le \underline V_L(z)\le \overline V(z)\le1,\qquad z\in\mathbb R,
\end{equation}
and possess the desired asymptotic behaviors at infinity
\begin{equation}\label{eq:barrier-limits-L}
\lim_{z\to-\infty}\underline V_L(z)=\lim_{z\to-\infty}\overline V(z)=1,\qquad \lim_{z\to+\infty}\underline V_L(z)=\lim_{z\to+\infty}\overline V(z)=0.
\end{equation}

To prove the existence of traveling wave solutions, we shall use the topology of local uniform convergence on $\mathcal K_L$. To this end, we introduce the weighted local norm
\begin{equation*}
\|v\|_*:=\sum_{k=1}^{\infty}2^{-k}\|v\|_{L^\infty([-k,k])}.
\end{equation*}
For sequences in $\mathcal K_L$, convergence with respect to $\|\cdot\|_*$ is equivalent to local uniform convergence on $\mathbb R$. More precisely, for any sequence $\{f_n\}\subset \mathcal K_L$ and any $f\in\mathcal K_L$,
\begin{equation*}
\|f_n-f\|_*\to0
\end{equation*}
if and only if
\begin{equation*}
f_n\to f
\end{equation*}
uniformly on every compact interval of $\mathbb R$.

Moreover, $\mathcal K_L$ is nonempty, bounded, closed and convex with respect to this topology. Indeed,
\begin{equation*}
\frac{\underline V_L+\overline V}{2}\in\mathcal K_L,
\end{equation*}
and hence $\mathcal K_L\neq\emptyset$. For every $v\in\mathcal K_L$, \eqref{eq:barrier-order-L} leads to
\begin{equation*}
0\le v(z)\le1,\qquad z\in\mathbb R,
\end{equation*}
so $\mathcal K_L$ is bounded in the sense that
\begin{equation*}
\|v\|_\infty\le1,\qquad \|v\|_*\le1,\qquad v\in\mathcal K_L.
\end{equation*}
The convexity of $\mathcal K_L$ follows directly from the pointwise order defining $\mathcal K_L$. Finally, if $v_n\in\mathcal K_L$ and $v_n\to v$ locally uniformly on $\mathbb R$, then, for every fixed $z\in\mathbb R$, passing to the limit in
\begin{equation*}
\underline V_L(z)\le v_n(z)\le\overline V(z)
\end{equation*}
yields
\begin{equation*}
\underline V_L(z)\le v(z)\le\overline V(z),\qquad z\in\mathbb R.
\end{equation*}
Moreover, by \eqref{eq:barrier-limits-L} and the squeeze theorem, $v$ has finite limits at $\pm\infty$, and hence $v$ is uniformly continuous on $\mathbb R$. Therefore $v\in\mathcal K_L$, and $\mathcal K_L$ is closed with respect to $\|\cdot\|_*$. \medskip

We now state our main existence result for traveling wave solutions.
\begin{theorem}\label{thm:main}
Assume that $d_1>1$ and that hypothesis ($H_1$) holds. Then for any $\theta\ge 2\sqrt{a_*}=2\sqrt{rD(0)}$, equation \eqref{eq:main-model} admits a traveling wave solution.
\end{theorem}

Furthermore, we establish the strict positivity and monotonicity properties of the traveling wave profile obtained in Theorem \ref{thm:main}.

\begin{theorem}\label{thm:pro}
Suppose that all assumptions in Theorem \ref{thm:main} hold, and let $(U,V,W)$ be the traveling wave solution to \eqref{eq:main-model} constructed in Theorem~\ref{thm:main}. Then
\begin{equation*}
0<U(z)<1,\qquad 0<V(z)<1,\qquad 0<W(z)<1,\qquad z\in\mathbb R,
\end{equation*}
and
\begin{equation*}
U'(z)>0,\qquad V'(z)<0,\qquad W'(z)<0,\qquad z\in\mathbb R.
\end{equation*}
\end{theorem}

In addition, the constructed wave has one-sided exponential convergence at both spatial infinities; a concise statement and its proof are given in Proposition~\ref{prop:asymptotic-estimates}.

\section{Construction of the map}\label{sec:construction}

In this section, for each $v\in\mathcal K_L$, we construct the associated functions $W_v$, $U_v$, and $V_v$, which will be used to define the fixed point map. We then analyze the crucial properties of $U_v$, $V_v$ and $W_v$, including uniform boundedness and long-time asymptotic behavior. This reduction transforms the original existence problem into a fixed point problem linked to the second equation in \eqref{eq:tw-system}.

For any $v\in\mathcal K_L$, define
\begin{equation}\label{eq:W-convolution}
W_v(z):=\int_{\mathbb R}G(y)v(z-y)\,dy,\qquad z\in\mathbb R,
\end{equation}
where
\begin{equation*}
G(y):= \frac{c}{\lambda_2^+-\lambda_2^-}
\begin{cases}
e^{\lambda_2^-y}, & y\geqslant0,\\[0.2cm]
e^{\lambda_2^+y}, & y<0,
\end{cases}
\end{equation*}
and $\lambda_2^\pm$ are given in \eqref{eq:lambda-pm}. Clearly,
\begin{equation*}
G(y)>0,\quad G\in C(\mathbb R)\cap L^1(\mathbb R),
\end{equation*}
and
\begin{equation}\label{eq:G-mass-one}
\int_{\mathbb R}G(y)\,dy=1.
\end{equation}
Since $v\in\mathcal K_L$, $v$ is not identically zero. Then combining with \eqref{eq:barrier-order-L}, we obtain
\begin{equation*}
0\le v(z)\le1,\qquad z\in\mathbb R.
\end{equation*}
Thus, for each fixed $z\in\mathbb R$, we have
\begin{equation}\label{eq:Wv-integrand-bound}
0\le G(y)v(z-y)\le G(y),\qquad y\in\mathbb R,
\end{equation}
and $G(y)v(z-y)$ is not identically zero. Since $G\in L^1(\mathbb R)$, it follows from \eqref{eq:Wv-integrand-bound} that, for each fixed $z\in\mathbb R$, $y\mapsto G(y)v(z-y)$ lies in $L^1(\mathbb R)$. Hence the integral in \eqref{eq:W-convolution} is well defined. With the above preparations, the function $W_v$ defined by \eqref{eq:W-convolution} satisfies the following properties.
\begin{lemma}\label{lem:Wv-representation}
For any $v\in\mathcal K_L$, let $W_v$ be defined in \eqref{eq:W-convolution}. Then $W_v\in C^2(\mathbb R)$ is the unique bounded solution to
\begin{equation}\label{eq:W-equation}
W_v''+\theta W_v'+c(v-W_v)=0,
\end{equation}
satisfying
\begin{equation}\label{eq:W-basic-ends}
W_v(-\infty)=1,\quad W_v(+\infty)=0.
\end{equation}
Moreover,
\begin{equation*}
0<W_v(z)<1,\qquad z\in\mathbb R.
\end{equation*}
There exists a constant $C_W=C_W(L)>0$, independent of $v$, such that
\begin{equation}\label{eq:W-tail}
0<W_v(z)\le C_We^{-\lambda z},\qquad|W_v'(z)|\le C_We^{-\lambda z},\qquad z\ge L.
\end{equation}
Furthermore, for any compact interval $I\subset\mathbb R$, there exists $C_{IW}>0$, independent of $v$, such that
\begin{equation*}
\|W_v\|_{C^2(I)}\le C_{IW}.
\end{equation*}
\end{lemma}

\begin{proof}
By the definition of $G$, we rewrite \eqref{eq:W-convolution} as
\begin{equation}\label{eq:Wv-expanded}
W_v(z)=I_1(z)+I_2(z),\qquad z\in\mathbb R,
\end{equation}
where
\begin{equation*}
I_1(z):=\frac{c}{\lambda_2^+-\lambda_2^-}\int_{-\infty}^{z}e^{\lambda_2^-(z-s)}v(s)\,ds,\qquad I_2(z):=\frac{c}{\lambda_2^+-\lambda_2^-}\int_z^{+\infty}e^{\lambda_2^+(z-s)}v(s)\,ds.
\end{equation*}
Since $v$ is bounded and uniformly continuous, differentiating \eqref{eq:Wv-expanded} yields
\begin{equation*}
W_v'(z)=\lambda_2^- I_1(z)+\lambda_2^+ I_2(z),
\end{equation*}
and
\begin{equation}\label{eq:Wv-second}
W_v''(z)=-cv(z)+(\lambda_2^-)^2I_1(z)+(\lambda_2^+)^2I_2(z).
\end{equation}
As $\lambda_2^\pm$ satisfy \eqref{eq:W-characteristic}, combining \eqref{eq:Wv-expanded}-\eqref{eq:Wv-second}, we get
\begin{equation*}
\begin{aligned}
W_v''+\theta W_v'-cW_v&=-cv+\bigl((\lambda_2^-)^2+\theta\lambda_2^- -c\bigr)I_1+\bigl((\lambda_2^+)^2+\theta\lambda_2^+ -c\bigr)I_2\\[0.2cm]
&=-cv.
\end{aligned}
\end{equation*}
Since $v\in\mathcal K_L$, \eqref{eq:barrier-limits-L} implies
\begin{equation*}
v(-\infty)=1,\qquad v(+\infty)=0.
\end{equation*}
For any fixed $y\in\mathbb R$,
\begin{equation*}
\lim_{z\to-\infty}v(z-y)=1,\qquad \lim_{z\to+\infty}v(z-y)=0.
\end{equation*}
Since $0\le v\le1$ and $G\in L^1(\mathbb R)$, the dominated convergence theorem gives
\begin{equation*}
\lim_{z\to-\infty}W_v(z)=\int_{\mathbb R}G(y)\,dy=1,\qquad \lim_{z\to+\infty}W_v(z)=0.
\end{equation*}
This verifies \eqref{eq:W-basic-ends}. Furthermore, since $v$, $I_1$ and $I_2$ are continuous, we conclude $W_v\in C^2(\mathbb R)$. Recalling that $G(y)v(z-y)$ does not vanish identically, \eqref{eq:G-mass-one} and \eqref{eq:Wv-integrand-bound}, we deduce
\begin{equation*}
0<W_v(z)=\int_{\mathbb R}G(y)v(z-y)\,dy<\int_{\mathbb R}G(y)\,dy=1.
\end{equation*}

We now turn to the uniqueness of such solutions. Let $W_1$ and $W_2$ be two bounded solutions to \eqref{eq:W-equation}. Define $Y:=W_1-W_2$. Then $Y$ satisfies the homogeneous equation
\begin{equation}\label{eq:W-homogeneous}
Y''+\theta Y'-cY=0.
\end{equation}
The characteristic equation associated with \eqref{eq:W-homogeneous} has exactly two roots $\lambda_2^-$ and $\lambda_2^+$, so the general solution reads
\begin{equation*}
Y(z)=C_1e^{\lambda_2^-z}+C_2e^{\lambda_2^+z}.
\end{equation*}
Because $\lambda_2^-<0<\lambda_2^+$, the term $e^{\lambda_2^-z}$ tends to infinity as $z\to-\infty$, while $e^{\lambda_2^+z}$ tends to infinity as $z\to+\infty$. As $Y$ is bounded on $\mathbb R$, we necessarily have $C_1=C_2=0$. Consequently, $Y\equiv 0$, namely, $W_1=W_2$. This establishes the uniqueness of bounded solutions to \eqref{eq:W-equation}.

It remains to establish the right-tail decay estimate. For $z\ge L$, using $v(s)\le 1$ for $s\le L$ and $v(s)\le e^{-\lambda(s-L)}$ for $s\ge L$, we obtain
\begin{align}\label{eq:I11-est}
I_1(z)&:=\frac{c}{\lambda_2^+-\lambda_2^-}\int_{-\infty}^{z}e^{\lambda_2^-(z-s)}v(s)\,ds \notag\\
&\le \frac{c}{\lambda_2^+-\lambda_2^-} \left(\int_{-\infty}^{L}e^{\lambda_2^-(z-s)}\,ds+\int_{L}^{z}e^{\lambda_2^-(z-s)}e^{-\lambda(s-L)}\,ds\right)\notag\\
&=\frac{c}{\lambda_2^+-\lambda_2^-}\left[-\frac{1}{\lambda_2^-}e^{\lambda_2^-(z-L)}+\frac{e^{-\lambda(z-L)}-e^{\lambda_2^-(z-L)}}{-\lambda_2^--\lambda}\right]\notag\\
&\le C_{W1}e^{-\lambda z},\qquad z\ge L,
\end{align}
where we have used $\lambda<-\lambda_2^-$.

Since $v\le\overline V$, we have, for $z\ge L$,
\begin{equation}\label{eq:v-exp-tail}
0\le v(z)\le e^{-\lambda(z-L)}=e^{\lambda L}e^{-\lambda z}.
\end{equation}
Substituting \eqref{eq:v-exp-tail} into the expression of $I_2$ yields
\begin{equation}\label{eq:I2-est}
I_2(z)\le \frac{ce^{\lambda L}}{\lambda_2^+-\lambda_2^-} \int_z^{+\infty}e^{\lambda_2^+(z-s)}e^{-\lambda s}\,ds =\frac{ce^{\lambda L}}{(\lambda_2^+-\lambda_2^-)(\lambda_2^++\lambda)}e^{-\lambda z}:=C_{W2}e^{-\lambda z}.
\end{equation}
Combining \eqref{eq:I11-est} and \eqref{eq:I2-est}, we obtain
\begin{equation*}
0<W_v(z)\le C_{W3}e^{-\lambda z},\qquad z\ge L,
\end{equation*}
where $C_{W3}=C_{W1}+C_{W2}$. By a similar argument, there is a constant $C_{W4}>0$ such that
\begin{equation*}
|W_v'(z)|\le C_{W4}e^{-\lambda z},\qquad z\ge L.
\end{equation*}
Setting $C_{W}=\max\{C_{W3},C_{W4}\}$, we thus obtain the desired estimate \eqref{eq:W-tail}.

Finally, let $I\subset\mathbb R$ be a compact interval. Since $0\le v\le1$, from \eqref{eq:Wv-expanded}-\eqref{eq:Wv-second}, we find a constant $C_{IW}>0$ such that
\begin{equation*}
\|W_v\|_{C^2(I)}\le C_{IW},
\end{equation*}
this proves the local $C^2$ estimate and completes the proof.

\end{proof}

Accordingly, in view of Lemma \ref{lem:Wv-representation}, from \eqref{eq:W-convolution}, we now introduce the first auxiliary map by
\begin{equation}\label{eq:T1-def}
T_1(v):=W_v,\qquad v\in\mathcal K_L.
\end{equation}
With the properties of $W_v$ established above, we now turn to the corresponding equation for $U_v$.

\begin{lemma}\label{lem:U-basic}
Assume that $d_1>1$. For each $v\in\mathcal K_L$, set $W_v=T_1(v)$ as in \eqref{eq:T1-def}. Then the equation
\begin{equation}\label{eq:Uv-problem}
\theta U_v'+U_v(1-U_v-d_1W_v)=0,\qquad z\in\mathbb R,
\end{equation}
subject to the boundary condition
\begin{equation}\label{eq:Uv-plus-bc}
U_v(+\infty)=1
\end{equation}
admits a unique solution $U_v\in C^1(\mathbb R)$. This solution satisfies
\begin{equation*}
0<U_v(z)<1,\ z\in\mathbb R,\qquad U_v(-\infty)=0,
\end{equation*}
and is given by
\begin{equation}\label{eq:Uv-formula}
U_v(z)=\theta\left(\int_z^{+\infty}e^{-\frac{1}{\theta}\int_z^s(1-d_1W_v(\tau))\,d\tau}\,ds\right)^{-1},\qquad z\in\mathbb R.
\end{equation}
Moreover, for every compact interval $I\subset\mathbb R$, there exists a constant $C_{IU}>0$, independent of $v$, such that
\begin{equation}\label{eq:U-C1}
\|U_v\|_{C^1(I)}\le C_{IU}.
\end{equation}
\end{lemma}

\begin{proof}
Since $W_v(+\infty)=0$, we have $1-d_1W_v(z)\to1$ as $z\to+\infty$, and hence there exists $R_1>0$ such that
\begin{equation*}
1-d_1W_v(y)\ge\frac12,\qquad y\ge R_1.
\end{equation*}
For each fixed $z\in\mathbb R$, set $R_z:=\max\{z,R_1\}$. Then, for $s\ge R_z$,
\begin{align*}
\int_z^s(1-d_1W_v(\tau))\,d\tau &=\int_z^{R_z}(1-d_1W_v(\tau))\,d\tau+ \int_{R_z}^s(1-d_1W_v(\tau))\,d\tau\\
&\ge \int_z^{R_z}(1-d_1W_v(\tau))\,d\tau +\frac12(s-R_z).
\end{align*}
Denote
\begin{equation*}
C_z:=\int_z^{R_z}(1-d_1W_v(\tau))\,d\tau.
\end{equation*}
Since $W_v$ is continuous and $[z,R_z]$ is a finite interval, $C_z$ is finite, and therefore, for $s\ge R_z$,
\begin{equation*}
e^{-\frac{1}{\theta}\int_z^s(1-d_1W_v(\tau))\,d\tau}\le e^{-\frac{C_z}{\theta}}e^{-\frac{s-R_z}{2\theta}},
\end{equation*}
it follows that
\begin{align*}
\int_{R_z}^{+\infty}e^{-\frac{1}{\theta}\int_z^s(1-d_1W_v(\tau))\,d\tau}\,ds &\le e^{-\frac{C_z}{\theta}}\int_{R_z}^{+\infty}e^{-\frac{s-R_z}{2\theta}}\,ds\\
&=2\theta e^{-\frac{C_z}{\theta}}<+\infty.
\end{align*}
Moreover, since the integrand is continuous on the finite interval $[z,R_z]$, we also have
\begin{equation*}
\int_z^{R_z}e^{-\frac{1}{\theta}\int_z^s(1-d_1W_v(\tau))\,d\tau}\,ds<+\infty.
\end{equation*}
Combining the two estimates with the positivity of the integrand gives
\begin{equation*}
0<\int_z^{+\infty}e^{-\frac{1}{\theta}\int_z^s(1-d_1W_v(\tau))\,d\tau}\,ds<+\infty.
\end{equation*}
Thus the right-hand side of \eqref{eq:Uv-formula} is well defined and positive. We define $U_v$ by \eqref{eq:Uv-formula}, or equivalently,
\begin{equation}\label{eq:Uv-reciprocal-formula}
\frac{1}{U_v(z)}=\frac{1}{\theta}\int_z^{+\infty}e^{-\frac{1}{\theta}\int_z^s(1-d_1W_v(\tau))\,d\tau}\,ds,\qquad z\in\mathbb R.
\end{equation}

We now verify that \eqref{eq:Uv-reciprocal-formula} indeed gives a positive solution. Since $W_v\in C^2(\mathbb R)$, differentiating \eqref{eq:Uv-reciprocal-formula} gives
\begin{equation}\label{eq:reciprocal-equation-short}
\left(\frac{1}{U_v}\right)'-\frac{1-d_1W_v}{\theta}\frac{1}{U_v}=-\frac{1}{\theta}.
\end{equation}
Indeed, the above differentiation is justified by the exponential decay of the integrand at $+\infty$, locally uniformly with respect to $z$. Since $U_v>0$, multiplying \eqref{eq:reciprocal-equation-short} by $-\theta U_v^2$ yields
\begin{equation*}
\theta U_v'+U_v(1-U_v-d_1W_v)=0,
\end{equation*}
which is exactly \eqref{eq:Uv-problem}. Hence $U_v\in C^1(\mathbb R)$.

It remains to verify the boundary condition \eqref{eq:Uv-plus-bc}. From \eqref{eq:Uv-reciprocal-formula}, we rewrite
\begin{equation*}
\frac{1}{U_v(z)}=\frac{1}{\theta}e^{\frac{1}{\theta}\int_0^z(1-d_1W_v(\tau))\,d\tau}\int_z^{+\infty}e^{-\frac{1}{\theta}\int_0^s(1-d_1W_v(\tau))\,d\tau}\,ds.
\end{equation*}
Equivalently,
\begin{equation*}
U_v(z)=\theta\frac{e^{-\frac{1}{\theta}\int_0^z(1-d_1W_v(\tau))\,d\tau}}{\int_z^{+\infty}e^{-\frac{1}{\theta}\int_0^s(1-d_1W_v(\tau))\,d\tau}\,ds}.
\end{equation*}
Since $1-d_1W_v(z)\to1$ as $z\to+\infty$, and by the estimate above the denominator is the tail of an integrable function, both the numerator and the denominator tend to zero as $z\to+\infty$. Applying L'Hospital's rule, we obtain
\begin{align*}
\lim_{z\to+\infty}U_v(z) &=\theta\lim_{z\to+\infty}\frac{e^{-\frac{1}{\theta}\int_0^z(1-d_1W_v(\tau))\,d\tau}}{\int_z^{+\infty}e^{-\frac{1}{\theta}\int_0^s(1-d_1W_v(\tau))\,d\tau}\,ds }\\
&=\theta\lim_{z\to+\infty}\frac{-\frac{1}{\theta}(1-d_1W_v(z))e^{-\frac{1}{\theta}\int_0^z(1-d_1W_v(\tau))\,d\tau}}{-e^{-\frac{1}{\theta}\int_0^z(1-d_1W_v(\tau))\,d\tau} }\\
&=\lim_{z\to+\infty}(1-d_1W_v(z))=1.
\end{align*}
Thus $U_v(+\infty)=1$.

Next, since $W_v(z)>0$ for every $z\in\mathbb R$, we have
\begin{equation*}
\int_z^{+\infty}e^{-\frac{1}{\theta}\int_z^s(1-d_1W_v(\tau))\,d\tau}\,ds> \int_z^{+\infty}e^{-\frac{s-z}{\theta}}\,ds=\theta.
\end{equation*}
By \eqref{eq:Uv-formula}, we obtain
\begin{equation*}
0<U_v(z)<1,\qquad z\in\mathbb R.
\end{equation*}
In particular, by hypothesis $(H_1)$,
\begin{equation*}
D(U_v(z))>0,\qquad z\in\mathbb R.
\end{equation*}

To prove the left-end limit, recall that $W_v(-\infty)=1$ and $d_1>1$. Then there exist constants $R_0<0$ and $\delta_0>0$ such that
\begin{equation*}
1-d_1W_v(z)\le-\delta_0,\qquad z\le R_0.
\end{equation*}
For $z<R_0$, it follows from \eqref{eq:Uv-reciprocal-formula} that
\begin{align*}
\frac{1}{U_v(z)} &\ge\frac{1}{\theta}\int_z^{R_0}e^{-\frac{1}{\theta}\int_z^s(1-d_1W_v(\tau))\,d\tau}\,ds\\
&\ge\frac{1}{\theta}\int_z^{R_0}e^{\frac{\delta_0}{\theta}(s-z)}\,ds =\frac{1}{\delta_0}\left(e^{\frac{\delta_0}{\theta}(R_0-z)}-1\right).
\end{align*}
Taking $z\to-\infty$, we get $\frac{1}{U_v(z)}\to+\infty$, and therefore
\begin{equation*}
U_v(-\infty)=0.
\end{equation*}

Finally, from \eqref{eq:Uv-problem}, together with $0<U_v<1$ and $0<W_v<1$, we have
\begin{equation*}
|U_v'(z)| =\frac{1}{\theta}U_v(z)|U_v(z)+d_1W_v(z)-1| \le\frac{1}{\theta}U_v(z)\bigl(U_v(z)+d_1W_v(z)+1\bigr) \le\frac{d_1+2}{\theta},
\end{equation*}
for all $z\in\mathbb R$. Hence, for every compact interval $I\subset\mathbb R$, there exists a constant $C_{IU}>0$, independent of $v$, such that
\begin{equation*}
\|U_v\|_{C^1(I)}\le C_{IU}.
\end{equation*}
This proves \eqref{eq:U-C1}.

It remains to prove uniqueness. Let $\widehat U_v\in C^1(\mathbb R)$ be another solution of \eqref{eq:Uv-problem} satisfying \eqref{eq:Uv-plus-bc}. Since $\widehat U_v(+\infty)=1$, we have $\widehat U_v>0$ for all sufficiently large $z$. Indeed, if $\widehat U_v$ had a finite zero, then the local uniqueness theorem for ordinary differential equations applied to \eqref{eq:Uv-problem} would imply $\widehat U_v\equiv0$, contradicting $\widehat U_v(+\infty)=1$. Hence $\widehat U_v(z)>0$ for all $z\in\mathbb R$.

For such a positive solution $\widehat U_v$, division by $\widehat U_v^2$ gives
\begin{equation*}
\left(\frac{1}{\widehat U_v}\right)' -\frac{1-d_1W_v}{\theta}\frac{1}{\widehat U_v} =-\frac{1}{\theta}.
\end{equation*}
Solving this linear equation on $[z,R]$ yields
\begin{equation*}
\frac{1}{\widehat U_v(z)} =\frac{1}{\widehat U_v(R)}e^{-\frac{1}{\theta}\int_z^R(1-d_1W_v(\tau))\,d\tau} +\frac{1}{\theta}\int_z^Re^{-\frac{1}{\theta}\int_z^s(1-d_1W_v(\tau))\,d\tau}\,ds.
\end{equation*}
Since $W_v(+\infty)=0$ and $\widehat U_v(+\infty)=1$, the first term on the right-hand side tends to zero as $R\to+\infty$. Letting $R\to+\infty$, we obtain
\begin{equation*}
\frac{1}{\widehat U_v(z)}=\frac{1}{\theta}\int_z^{+\infty}e^{-\frac{1}{\theta}\int_z^s(1-d_1W_v(\tau))\,d\tau}\,ds.
\end{equation*}
This coincides with \eqref{eq:Uv-reciprocal-formula}; hence $\frac{1}{\widehat U_v}=\frac{1}{U_v}$ on $\mathbb R$. Since both functions are positive, we conclude that $\widehat U_v\equiv U_v$, and uniqueness follows.
\end{proof}

Consequently, by virtue of ($H_1$), we deduce that $D(U_v(z))>0$ for $z\in\mathbb{R}$ and define the second auxiliary map by
\begin{equation*}
T_2(W_v):=U_v,\qquad v\in\mathcal K_L.
\end{equation*}

From the above statements, we have already constructed two maps derived from $v\in\mathcal K_L$. We further establish the third map from the second equation of \eqref{eq:tw-system}, taking $U_v$ as the known function. Inspired by \cite{ChenQi2025}, the map in the fixed point theorem consists of the composition of these three maps. The existence of traveling wave solutions to \eqref{eq:main-model} will follow from the fixed point theorem. In order to finish it, we introduce the change of variable
\begin{equation}\label{defeta}
\eta(z):=-\int_0^z\frac{ds}{D(U_v(s))}, \qquad z\in\mathbb R,
\end{equation}
as in \cite{gm22,ChenQi2025}. By Lemma \ref{lem:U-basic} and ($H_1$),
\begin{equation*}
\eta'(z)=-\frac{1}{D(U_v(z))}<0,
\end{equation*}
then the map $z\mapsto\eta(z)$ is strictly decreasing, and admits an inverse denoted by $z=Z_v(\eta)$; furthermore,
\begin{equation*}
\eta(+\infty)=-\infty,\qquad \eta(-\infty)=+\infty.
\end{equation*}
Moreover, also by Lemma \ref{lem:U-basic} and ($H_1$), we have $D(U_v(z))\le D(0)$, which implies
\begin{equation}\label{eq:eta-growth}
\eta(z)\le -\frac{z}{D(0)}\quad (z\ge0),\qquad \eta(z)\ge -\frac{z}{D(0)}\quad (z\le0).
\end{equation}

We now rewrite the second equation of \eqref{eq:tw-system} in the new variable. Define
\begin{equation*}
\widetilde U_v(\eta):=U_v(Z_v(\eta)),\quad\widetilde V(\eta):=V(Z_v(\eta)),
\end{equation*}
which satisfy
\begin{equation}\label{eq:Vtilde-equation}
\widetilde V_{\eta\eta}-\theta\widetilde V_\eta+a_v(\eta)\widetilde V(1-\widetilde V)=0,\qquad \eta\in\mathbb R,
\end{equation}
where
\begin{equation*}
a_v(\eta):=rD(\widetilde U_v(\eta)).
\end{equation*}
By Lemma \ref{lem:U-basic} and ($H_1$), direct computations give
\begin{equation*}
a_v\in C^1(\mathbb{R}),\qquad0<a_v(\eta)\le a_*,\qquad a_v(-\infty)=0,\qquad a_v(+\infty)=a_*.
\end{equation*}

In contrast with the study in \cite{ChenQi2025}, we do not directly prove the existence of solutions to the boundary value problem for the nonautonomous second-order ordinary differential equation \eqref{eq:Vtilde-equation}. Instead, we consider an initial value problem for an auxiliary parabolic equation. The solution of this initial value problem converges to a limit as time tends to infinity, and this limit is exactly the solution to the aforementioned boundary value problem for \eqref{eq:Vtilde-equation}. Now, we introduce the initial value problem for an auxiliary parabolic equation,
\begin{equation}\label{eq:aux-parabolic}
\begin{cases}
(\mathcal V_v)_t=(\mathcal V_v)_{\eta\eta}-\theta(\mathcal V_v)_\eta+a_v(\eta)\mathcal V_v(1-\mathcal V_v), & \eta\in\mathbb R,\ t>0,\\[0.2cm]
\mathcal V_v(\eta,0)=\mathcal V_0(\eta), & \eta\in\mathbb R,
\end{cases}
\end{equation}
where $\mathcal V_0\in C(\mathbb{R};[0,1])$.

We first establish a local Schauder estimate for \eqref{eq:aux-parabolic}.

\begin{lemma}\label{lem:schauder}
Assume that ($H_1$) holds and $d_1>1$. Let $\mathcal V_v$ be a solution to \eqref{eq:aux-parabolic}. Then, for any $R>0$ and $0<\tau<T<\infty$, there exist $\beta\in(0,1)$ and $C=C(R,\tau,T,a_v(\eta))>0$, independent of time translations, such that
\begin{equation*}
\|\mathcal V_v\|_{C^{2+\beta,1+\frac{\beta}{2}}((-R,R)\times(\tau,T))}\le C.
\end{equation*}
\end{lemma}
\begin{proof}
By the properties of $a_v(\eta)$ and the comparison theorem, we have $0\leq\mathcal V_v\leq1$. Then for $R>0$ and $0<\tau<T<\infty$, we apply the local parabolic $L^p$ estimate with $p>\frac32$ \cite[estimate (10.12), Section 10, Chapter IV]{Ladyzhenskaya}, then there exist $R_0$ and $\tau_0$ satisfying
\begin{equation*}
R<R_0,\qquad 0<\tau_0<\tau,
\end{equation*}
such that
\begin{equation*}
\|\mathcal V_v\|_{W^{2,1}_p((-R_0,R_0)\times(\tau_0,T))}\le C_1,
\end{equation*}
where $C_1$ depends only on $R,\tau,T,M$, with $M=\max\limits_{[-R_0,R_0]}a_v$. The parabolic Sobolev embedding theorem \cite[Lemma 3.3, Chapter II]{Ladyzhenskaya} further gives some $\beta\in(0,1)$ such that
\begin{equation}\label{eq:holder-V-before-schauder}
\|\mathcal V_v\|_{C^{\beta,\frac{\beta}{2}}((-R_0,R_0)\times(\tau_0,T))}\le C_2.
\end{equation}
After reducing the value of $\beta$ if necessary, the smoothness of $a_v$ combined with \eqref{eq:holder-V-before-schauder} yields
\begin{equation*}
a_v(\eta)\mathcal V_v(1-\mathcal V_v)\in C^{\beta,\frac{\beta}{2}}((-R_0,R_0)\times(\tau_0,T)),
\end{equation*}
namely,
\begin{equation*}
\|a_v\mathcal V_v(1-\mathcal V_v)\|_{C^{\beta,\frac{\beta}{2}}((-R_0,R_0)\times(\tau_0,T))} \le C(R,\tau,T,M_{R_0,\beta}),
\end{equation*}
where
\begin{equation*}
M_{R_0,\beta}:=\|a_v\|_{C^\beta([-R_0,R_0])}=\|a_v\|_{L^\infty([-R_0,R_0])}+\sup_{\substack{\eta_1,\eta_2\in[-R_0,R_0]\\ \eta_1\neq\eta_2}}\frac{|a_v(\eta_1)-a_v(\eta_2)|}{|\eta_1-\eta_2|^\beta}.
\end{equation*}
Finally, applying the local Schauder estimate \cite[Theorem 10.1,Chapter IV]{Ladyzhenskaya} to the equation in \eqref{eq:aux-parabolic}, we obtain
\begin{equation*}
\|\mathcal V_v\|_{C^{2+\beta,1+\frac{\beta}{2}}((-R,R)\times(\tau,T))}\le C.
\end{equation*}

The above argument is valid for every time-translated solution $\mathcal V_v^\ell(\eta,t):=\mathcal V_v(\eta,t+\ell)$, with $\ell\ge0$, which follows from the time-translation invariance of equation \eqref{eq:aux-parabolic}. Hence the constant $C$ is independent of time translations. This completes the proof.
\end{proof}

For each $v\in\mathcal K_L$, the following lemma proves that \eqref{eq:aux-parabolic} admits a steady-state solution connecting the states $0$ and $1$. This enables us to define the third map.

We first prepare the super- and sub-solutions used below. By the definition of $\overline V$, $v(z)\le e^{-\lambda z}$ for $z\in\mathbb R$. Since $\lambda<-\lambda_2^-$, we have $\int_{\mathbb R}G(y)e^{\lambda y}\,dy<+\infty$. Hence the Green representation gives
\begin{equation*}
W_v(z)\le Ce^{-\lambda z},\qquad z\in\mathbb R,
\end{equation*}
where $C>0$ is independent of $v$ and $L$. In the following computations, the value of $C$ may change from line to line. Taking the inverse of \eqref{eq:Uv-formula} and setting $t=s-z$, we obtain
\begin{equation*}
\frac{1}{U_v(z)} =\frac{1}{\theta}\int_0^{+\infty}e^{-\frac{t}{\theta}} \exp\left(\frac{d_1}{\theta}\int_0^tW_v(z+s)\,ds\right)dt,
\end{equation*}
and therefore
\begin{equation}\label{eq:U-right-uniform}
0<1-U_v(z)\le Ce^{-\lambda z},\qquad D(U_v(z))\le Ce^{-\lambda z},\qquad z\ge0.
\end{equation}
The change of variables \eqref{defeta} gives
\begin{equation}\label{eq:av-identity}
\int_{\eta(z)}^0a_v(s)\,ds=rz,\qquad z\in\mathbb R,
\end{equation}
and hence
\begin{equation}\label{eq:av-integrability}
\int_{-\infty}^0a_v(s)\,ds=+\infty,\qquad \int_{-\infty}^0a_v(s)^2\,ds =r^2\int_0^{+\infty}D(U_v(z))\,dz\le C.
\end{equation}

We construct entire solutions $b_v^+$ and $b_v^-$ satisfying
\begin{equation}\label{eq:bpm}
\begin{cases}
(b_v^+)'=\theta b_v^+-(b_v^+)^2-a_v, \qquad 0<b_v^+\le\lambda_1^-,\\[0.15cm]
(b_v^-)'=\theta b_v^-+(b_v^-)^2-a_v, \qquad 0<b_v^-\le\nu_*.
\end{cases}
\end{equation}
For each $T\in\mathbb R$, $b^\pm_v(\eta)$ are obtained by solving the two equations backwards with terminal conditions $b_{v,T}^{\pm}(T)=0$. In the reversed independent variable, the right-hand sides are $-\theta b+b^2+a_v$ and $-\theta b-b^2+a_v$, respectively. Both equal $a_v>0$ at $b=0$, while at their respective upper endpoints they satisfy
\begin{equation*}
-\theta\lambda_1^-+(\lambda_1^-)^2+a_v=a_v-a_*\le0, \qquad -\theta\nu_*-\nu_*^2+a_v=a_v-a_*\le0.
\end{equation*}
Thus $[0,\lambda_1^-]$ and $[0,\nu_*]$ are invariant under backward integration, and the solutions extend to $(-\infty,T]$ and are positive for $\eta<T$. If $T_2>T_1$, then $b_{v,T_2}^{\pm}(T_1)>0=b_{v,T_1}^{\pm}(T_1)$, so scalar comparison gives $b_{v,T_2}^{\pm}>b_{v,T_1}^{\pm}$ on $(-\infty,T_1]$. The bounded monotone limits as $T\to+\infty$, together with the local derivative bounds derived from the equations, therefore yield nonnegative entire solutions $b_v^\pm$ in the same intervals, with convergence in $C^1_{\mathrm{loc}}(\mathbb R)$.

To establish strict positivity and the required estimate, write
\begin{equation*}
(b_v^+)'-(\theta-b_v^+)b_v^+=-a_v, \qquad (b_v^-)'-(\theta+b_v^-)b_v^-=-a_v.
\end{equation*}
For $R>\eta$, the integrating-factor formula gives
\begin{align*}
b_v^\pm(\eta) &=\exp\left(-\int_\eta^R(\theta\mp b_v^\pm(\tau))\,d\tau\right)b_v^\pm(R)\\
&\quad+\int_\eta^R \exp\left(-\int_\eta^s(\theta\mp b_v^\pm(\tau))\,d\tau\right)a_v(s)\,ds.
\end{align*}
Recalling that $b_v^\pm$ are bounded and that $\theta-b_v^+\ge\theta-\lambda_1^->0$ and $ \theta+b_v^-\ge\theta>0$, we can take the limit $R\to\infty$ to obtain
\begin{equation*}
b_v^\pm(\eta)=\int_\eta^{+\infty}\exp\left(-\int_\eta^s \bigl(\theta\mp b_v^\pm(\tau)\bigr)\,d\tau\right) a_v(s)\,ds>0.
\end{equation*}
The strict positivity follows from $a_v(s)>0$ for every finite $s$. Both kernels are bounded above by $e^{-(\theta-\lambda_1^-)(s-\eta)}$; hence, with $s=\eta+t$,
\begin{equation}\label{eq:bpm-bound}
0<b_v^{\pm}(\eta) \le\int_0^{+\infty}e^{-(\theta-\lambda_1^-)t}a_v(\eta+t)\,dt.
\end{equation}
Since $a_v(-\infty)=0$ and $a_v\le a_*$, dominated convergence implies $b_v^{\pm}(-\infty)=0$.

For the limits at $+\infty$, fix $0<\varepsilon<\min\{\lambda_1^-,\nu_*\}$ and choose $R$ so that $a_v(\eta)\ge a_*-\frac{\varepsilon^2}{2}$ for $\eta\ge R$. Using $\lambda_1^-(\theta-\lambda_1^-)=a_*$, $\nu_*(\theta+\nu_*)=a_*$, and $2\lambda_1^-\le\theta$, we obtain
\begin{equation*}
\begin{aligned}
(b_v^+)'&\le-\frac{\varepsilon^2}{2} &&\text{whenever }\eta\ge R\text{ and }b_v^+\le\lambda_1^--\varepsilon,\\
(b_v^-)'&\le-\frac{\varepsilon^2}{2} &&\text{whenever }\eta\ge R\text{ and }b_v^-\le\nu_*-\varepsilon.
\end{aligned}
\end{equation*}
If either solution were at or below the corresponding level at some $\eta_0\ge R$, it could not cross that level upwards and would reach zero in finite forward distance, contradicting strict positivity. Therefore $b_v^+>\lambda_1^--\varepsilon$ and $b_v^->\nu_*-\varepsilon$ on $[R,+\infty)$. Since $\varepsilon>0$ is arbitrary,
\begin{equation*}
\liminf_{\eta\to+\infty}b_v^+(\eta)\ge\lambda_1^-,\qquad \liminf_{\eta\to+\infty}b_v^-(\eta)\ge\nu_*.
\end{equation*}
Together with $b_v^+\le\lambda_1^-$ and $b_v^-\le\nu_*$, we conclude that
\begin{equation}\label{eq:bpm-limits}
b_v^+(+\infty)=\lambda_1^-, \qquad b_v^-(+\infty)=\nu_*.
\end{equation}

These arguments apply to every nonnegative entire solution in the respective intervals. For two such solutions $b_1,b_2$ of the first equation, their difference $d=b_1-b_2$ satisfies
\begin{equation*}
d'=(\theta-b_1-b_2)d,\qquad \theta-b_1-b_2\ge\theta-2\lambda_1^-\ge0.
\end{equation*}
Thus $|d|$ is nondecreasing, whereas \eqref{eq:bpm-limits} gives $d(+\infty)=0$; hence $d\equiv0$. For the second equation, the difference satisfies $d'=(\theta+b_1+b_2)d$, and the same argument proves uniqueness. In particular, both the construction and uniqueness remain valid when $\theta=2\sqrt{a_*}$.

Set $\delta_v:=b_v^+-b_v^-$. Subtracting the equations in \eqref{eq:bpm} yields
\begin{equation*}
\delta_v'-\theta\delta_v =-\bigl((b_v^+)^2+(b_v^-)^2\bigr).
\end{equation*}
Multiplying by the integrating factor and integrating from $\eta$ to $R$ gives
\begin{equation*}
\delta_v(\eta)=e^{-\theta(R-\eta)}\delta_v(R) +\int_\eta^R e^{-\theta(s-\eta)} \bigl((b_v^+(s))^2+(b_v^-(s))^2\bigr)\,ds.
\end{equation*}
Since $\delta_v$ is bounded, letting $R\to+\infty$ yields
\begin{equation}\label{eq:delta-positive}
\delta_v(\eta)=\int_0^{+\infty}e^{-\theta t} \left((b_v^+(\eta+t))^2+(b_v^-(\eta+t))^2\right)dt>0.
\end{equation}

By \eqref{eq:bpm-bound}, the Cauchy--Schwarz inequality gives
\begin{equation*}
(b_v^\pm(\eta))^2\le\frac{1}{\theta-\lambda_1^-} \int_0^{+\infty}e^{-(\theta-\lambda_1^-)t}a_v(\eta+t)^2\,dt.
\end{equation*}
Since \eqref{eq:av-integrability} and $a_v\le a_*$ imply $\int_{-\infty}^t a_v(s)^2\,ds\le C+a_*^2t$ for $t\ge0$, Fubini's theorem yields
\begin{align*}
\int_{-\infty}^0(b_v^\pm(\eta))^2\,d\eta &\le\frac{1}{\theta-\lambda_1^-} \int_0^{+\infty}e^{-(\theta-\lambda_1^-)t} \int_{-\infty}^t a_v(s)^2\,ds\,dt\\
&\le\frac{1}{\theta-\lambda_1^-} \int_0^{+\infty}e^{-(\theta-\lambda_1^-)t}(C+a_*^2t)\,dt\le C.
\end{align*}
Applying Fubini's theorem to \eqref{eq:delta-positive} then gives
\begin{align*}
0<\int_{-\infty}^0\delta_v(\eta)\,d\eta &=\frac{1}{\theta}\int_{-\infty}^0 \bigl((b_v^+(s))^2+(b_v^-(s))^2\bigr)\,ds\\
&\quad+\frac{1}{\theta}\int_0^{+\infty}e^{-\theta s} \bigl((b_v^+(s))^2+(b_v^-(s))^2\bigr)\,ds\le C,
\end{align*}
where the constants are independent of $v$ and $L$. Define
\begin{equation}\label{eq:HKM-def}
\begin{aligned}
H_v(\eta)&:=e^{-\frac{\lambda_1^-}{\theta}} \exp\left(\int_0^\eta b_v^+(s)\,ds\right),\\[0.15cm]
K_v(\eta)&:=\exp\left(-\int_0^\eta b_v^-(s)\,ds\right),\\[0.15cm]
M_v&:=e^{\frac{\lambda_1^-}{\theta}} \exp\left(\int_{-\infty}^0\delta_v(s)\,ds\right).
\end{aligned}
\end{equation}
Then $M_v\le M_0$ for some $M_0>0$ independent of $v$ and $L$. Since $H_v'=b_v^+H_v$ and $K_v'=-b_v^-K_v$, substitution of \eqref{eq:bpm} gives
\begin{equation}\label{eq:HK-equations}
H_v''-\theta H_v'+a_vH_v=0, \qquad K_v''-\theta K_v'-a_vK_v=0.
\end{equation}
For any $A<0$, integrating the second equation in \eqref{eq:bpm} over $[A,0]$ gives
\begin{align*}
\int_A^0a_v(\eta)\,d\eta &=\theta\int_A^0b_v^-(\eta)\,d\eta +\int_A^0(b_v^-(\eta))^2\,d\eta-b_v^-(0)+b_v^-(A)\\
&\le(\theta+\nu_*)\int_A^0b_v^-(\eta)\,d\eta+\nu_*.
\end{align*}
Since $\int_{-\infty}^0a_v=+\infty$ by \eqref{eq:av-integrability}, letting $A\to-\infty$ shows that $\int_{-\infty}^0b_v^-=+\infty$. The inequality $b_v^+>b_v^-$ from \eqref{eq:delta-positive} also gives $\int_{-\infty}^0b_v^+=+\infty$. Thus \eqref{eq:HKM-def} yields the limits at $-\infty$, while \eqref{eq:bpm-limits} gives those at $+\infty$:
\begin{equation}\label{eq:HK-limits}
H_v(-\infty)=0,\qquad K_v(-\infty)=+\infty, \qquad H_v(+\infty)=+\infty,\qquad K_v(+\infty)=0.
\end{equation}

We define the candidate super-solution and sub-solution by
\begin{equation}\label{eq:super}
\overline\Phi_v(\eta):=\frac{H_v(\eta)}{1+H_v(\eta)},\qquad \eta\in\mathbb R,
\end{equation}
and
\begin{equation}\label{eq:lower-barrier-def}
\underline\Phi_v(\eta):=\frac{1}{1+M_vK_v(\eta)},\qquad \eta\in\mathbb R.
\end{equation}
Since
\begin{equation}\label{eq:barrier-product}
M_vH_v(\eta)K_v(\eta) =\exp\left(\int_{-\infty}^{\eta}\delta_v(s)\,ds\right)>1,
\end{equation}
we have
\begin{equation}\label{eq:barrier-order-whole}
0<\underline\Phi_v(\eta)<\overline\Phi_v(\eta)<1, \qquad \eta\in\mathbb R,
\end{equation}
and both functions tend to $0$ at $-\infty$ and to $1$ at $+\infty$. Moreover,
\begin{equation}\label{eq:common-tail}
\lim_{\eta\to-\infty}\frac{\underline\Phi_v(\eta)}{H_v(\eta)} =\lim_{\eta\to-\infty}\frac{\overline\Phi_v(\eta)}{H_v(\eta)}=1.
\end{equation}

With the above super- and sub-solutions, we now establish the stationary profile.
\begin{lemma}\label{lem:stationary-profile}
Assume that $d_1>1$ and that hypothesis ($H_1$) holds. For every $v\in\mathcal K_L$, let $\mathcal V_v(\eta,t)$ be the solution of \eqref{eq:aux-parabolic} with initial data
\begin{equation*}
\mathcal V_v(\eta,0)=\overline\Phi_v(\eta).
\end{equation*}
Then there exists a stationary profile $\widetilde V_v\in C^2(\mathbb R)$ satisfying
\begin{equation}\label{eq:stationary-profile}
\widetilde V_v''-\theta\widetilde V_v' +a_v(\eta)\widetilde V_v(1-\widetilde V_v)=0, \qquad \eta\in\mathbb R,
\end{equation}
with
\begin{equation}\label{eq:260909a}
0<\widetilde V_v(\eta)<1,\qquad \widetilde V_v(-\infty)=0,\qquad \widetilde V_v(+\infty)=1.
\end{equation}
Moreover,
\begin{equation}\label{eq:left-tail-eta}
\lim_{\eta\to-\infty}\frac{\widetilde V_v(\eta)}{H_v(\eta)}=1.
\end{equation}
For every $0<\rho<\nu_*$, there exists a constant $C_R=C_R(\rho,v)>0$ such that
\begin{equation}\label{eq:right-tail-eta}
0<1-\widetilde V_v(\eta)\le C_Re^{-\rho\eta}, \qquad \eta\ge0.
\end{equation}
The profile is unique among positive solutions satisfying \eqref{eq:260909a} and \eqref{eq:left-tail-eta}.
\end{lemma}

\begin{proof}
For any $v\in\mathcal K_L$, let
\begin{equation*}
\mathscr L_v[\phi] :=\phi_t-\phi_{\eta\eta}+\theta\phi_{\eta}-a_v(\eta)\phi(1-\phi).
\end{equation*}
Using \eqref{eq:HK-equations}, direct computations give
\begin{equation*}
\mathscr L_v[\underline\Phi_v] =-\frac{2M_v^2(K_v')^2}{(1+M_vK_v)^3}<0, \qquad \mathscr L_v[\overline\Phi_v] =\frac{2(H_v')^2}{(1+H_v)^3}>0.
\end{equation*}
Thus $\overline\Phi_v$ and $\underline\Phi_v$ are ordered super- and sub-solutions of \eqref{eq:aux-parabolic} on the whole real line. Applying the comparison principle to \eqref{eq:aux-parabolic} with initial datum $\mathcal V_v(\eta,0)=\overline\Phi_v(\eta)$, we obtain
\begin{equation}\label{eq:aux-flow-order}
\underline\Phi_v(\eta) \le \mathcal V_v(\eta,t) \le \overline\Phi_v(\eta), \qquad \eta\in\mathbb R,\quad t>0.
\end{equation}
Since $\overline\Phi_v$ is a super-solution, the comparison principle, as in \cite[Proposition 4.2]{2021 Lijing wangzhian}, yields
\begin{equation*}
\mathcal V_v(\eta,t+s)\le \mathcal V_v(\eta,t), \qquad \eta\in\mathbb R,\quad t,s>0.
\end{equation*}
Consequently, the pointwise limit
\begin{equation}\label{eq:selected-limit}
\widetilde V_v(\eta):=\lim_{t\to+\infty}\mathcal V_v(\eta,t)
\end{equation}
exists for every $\eta\in\mathbb R$.

We next show that this limit is a stationary solution. Let $t_n\to+\infty$ and define
\begin{equation*}
\mathcal V_v^n(\eta,t):=\mathcal V_v(\eta,t_n+t), \qquad \eta\in\mathbb R,\; t>0.
\end{equation*}
By Lemma~\ref{lem:schauder}, for every compact interval $I\subset\mathbb R$ and every $0<\tau<T<+\infty$, the family $\{\mathcal V_v^n\}$ is uniformly bounded in $C^{2+\beta,1+\frac{\beta}{2}}(I\times[\tau,T])$. Hence, by the Arzel\`a--Ascoli theorem, a subsequence converges locally in $C^{2,1}_{\mathrm{loc}}(\mathbb R\times(0,+\infty))$ to some function $\mathcal V_v^\infty$. For each fixed $t>0$, since $t_n+t\to+\infty$, the monotone convergence \eqref{eq:selected-limit} gives
\begin{equation*}
\mathcal V_v^\infty(\eta,t)=\lim_{n\to\infty}\mathcal V_v(\eta,t_n+t)=\widetilde V_v(\eta), \qquad \eta\in\mathbb R.
\end{equation*}
Thus $\mathcal V_v^\infty$ is independent of $t$. Passing to the limit in \eqref{eq:aux-parabolic}, we obtain
\begin{equation*}
0=(\widetilde V_v)''-\theta(\widetilde V_v)'+a_v(\eta)\widetilde V_v(1-\widetilde V_v),
\end{equation*}
which is exactly \eqref{eq:stationary-profile}.

The estimate \eqref{eq:aux-flow-order}, together with \eqref{eq:HK-limits}, yields the two end states and $0\le\widetilde V_v\le1$. The strong maximum principle gives
\begin{equation*}
0<\widetilde V_v(\eta)<1,\qquad \eta\in\mathbb R.
\end{equation*}
The common-tail relation \eqref{eq:common-tail}, \eqref{eq:aux-flow-order}, and the squeeze theorem give \eqref{eq:left-tail-eta}. For $0<\rho<\nu_*$, \eqref{eq:bpm-limits} gives an $R>0$ such that $b_v^-(\eta)\ge\rho$ for $\eta\ge R$. Since $\underline\Phi_v\le\widetilde V_v$,
\begin{equation*}
1-\widetilde V_v(\eta) \le M_vK_v(\eta) \le M_vK_v(R)e^{-\rho(\eta-R)}, \qquad \eta\ge R.
\end{equation*}
After enlarging the constant on $[0,R]$, this proves \eqref{eq:right-tail-eta}.

Finally, let $q_1,q_2$ be two positive solutions satisfying the same end states and \eqref{eq:left-tail-eta}. Their ratio tends to $1$ at both infinities. If $k=\mathop{\sup}\limits_{\eta\in\mathbb R}\frac{q_1(\eta)}{q_2(\eta)}>1$, the supremum is attained at some finite $\eta_0$. Setting $h=q_1-kq_2$, we have $h(\eta_0)=h'(\eta_0)=0$ and $h''(\eta_0)\le0$, whereas the two stationary equations give
\begin{equation*}
h''(\eta_0)=a_v(\eta_0)k(k-1)q_2(\eta_0)^2>0,
\end{equation*}
a contradiction. Reversing the roles of $q_1$ and $q_2$ proves uniqueness. This completes the proof.
\end{proof}

\section{Existence of traveling wave solutions}\label{sec:existence}

From Lemma \ref{lem:stationary-profile}, we obtain the third map $V_v=T_3(U_v)$. Accordingly, similar to \cite{ChenQi2025}, for any $v\in\mathcal K_L$, we define the composite map $F = T_3\circ T_2\circ T_1:\mathcal K_L\to X$ via
\begin{equation*}
F(v)(z):=\widetilde V_v(\eta(z)),\qquad z\in\mathbb R.
\end{equation*}
Our next goal is to prove that $F$ admits a fixed point inside $\mathcal K_L$.

\begin{lemma}\label{lem:F-invariant}
Assume all hypotheses of Lemma~\ref{lem:stationary-profile} hold. Let $M_0$ be the uniform bound obtained above and set
\begin{equation*}
C_*:=M_0e^{\frac{\nu_*}{\theta+\nu_*}}.
\end{equation*}
Choose $L>0$ so large that
\begin{equation*}
L\ge \frac{D(0)}{\nu_*}\max\{\ln C_*,0\},
\end{equation*}
then
\begin{equation}\label{eq:F-invariant}
F(\mathcal K_L)\subset \mathcal K_L.
\end{equation}
\end{lemma}
\begin{proof}
For any $v\in\mathcal K_L$, by Lemma~\ref{lem:stationary-profile}, we have
\begin{equation}\label{eq:F-strict-bounds}
0<F(v)(z)<1,\qquad z\in\mathbb R.
\end{equation}
We only need to verify $\underline V_L(z)\le F(v)(z)\le \overline V(z)$.

For $z\ge0$, we have $\eta(z)\le0$. Integrating the first equation in \eqref{eq:bpm} over $[\eta(z),0]$ and using \eqref{eq:av-identity},
\begin{equation*}
\theta\int_{\eta(z)}^0b_v^+(s)\,ds =b_v^+(0)-b_v^+(\eta(z)) +\int_{\eta(z)}^0(b_v^+(s))^2\,ds+rz \ge rz-\lambda_1^-.
\end{equation*}
Hence, by \eqref{eq:super} and \eqref{eq:HKM-def},
\begin{equation}\label{eq:F-right-uniform-new}
F(v)(z)\le \overline\Phi_v(\eta(z))\le H_v(\eta(z)) \le e^{-\frac{rz}{\theta}},\qquad z\ge0.
\end{equation}
Since $\lambda<\frac{r}{\theta}$, this gives
\begin{equation*}
F(v)(z)\le e^{-\lambda z}=\overline V(z),\qquad z\ge0.
\end{equation*}
For $z<0$, $\overline V(z)=1$, and the same upper bound follows from \eqref{eq:F-strict-bounds}.

If $z\le0$, then $\eta(z)\ge0$. Integrating the second equation in \eqref{eq:bpm} over $[0,\eta(z)]$ gives
\begin{align*}
-rz=\int_0^{\eta(z)}a_v(s)\,ds &=\theta\int_0^{\eta(z)}b_v^-(s)\,ds +\int_0^{\eta(z)}(b_v^-(s))^2\,ds\\
&\quad-b_v^-(\eta(z))+b_v^-(0)\\
&\le(\theta+\nu_*)\int_0^{\eta(z)}b_v^-(s)\,ds+\nu_*.
\end{align*}
Using \eqref{eq:lower-barrier-def}, we obtain
\begin{align}
1-F(v)(z) &\le M_vK_v(\eta(z))\notag\\
&\le M_0\exp\left(\frac{\nu_*+rz}{\theta+\nu_*}\right) =C_*e^{\frac{\nu_*z}{D(0)}},\qquad z\le0, \label{eq:F-left-uniform-new}
\end{align}
where $\nu_*(\theta+\nu_*)=a_*=rD(0)$. For $z<-L$, the choice $\mu<\frac{\nu_*}{D(0)}$ gives
\begin{equation*}
\frac{C_*e^{\frac{\nu_*z}{D(0)}}}{e^{\mu(z+L)}} \le C_*e^{-\frac{\nu_*L}{D(0)}}\le1,
\end{equation*}
and therefore
\begin{equation*}
F(v)(z)\ge1-e^{\mu(z+L)}=\underline V_L(z),\qquad z<-L.
\end{equation*}
For $z\ge-L$, since $\underline V_L(z)=0$ and $F(v)(z)>0$ by \eqref{eq:F-strict-bounds}, we also have
\begin{equation*}
F(v)(z)\ge \underline V_L(z),\qquad z\ge-L.
\end{equation*}
This proves \eqref{eq:F-invariant}.
\end{proof}

We now show that the map $F$ admits a fixed point in $\mathcal K_L$.

\begin{lemma}\label{lem:F-compact-continuous}
Assume all hypotheses of Lemmas \ref{lem:Wv-representation}--\ref{lem:stationary-profile} and Lemma \ref{lem:F-invariant} hold. The map $F:\mathcal K_L\to\mathcal K_L$ is continuous with respect to the norm $\|\cdot\|_*$, and $F(\mathcal K_L)$ is relatively compact with respect to $\|\cdot\|_*$. Moreover, $\mathcal K_L$ is nonempty, closed and convex with respect to this topology. Consequently, there exists some $V\in\mathcal K_L$ satisfying
\begin{equation*}
F(V)=V.
\end{equation*}
\end{lemma}

\begin{proof}
We first prove compactness. Take any sequence $\{v_n\}\subset\mathcal K_L$. The elements of this sequence are uniformly bounded due to the definition of $\mathcal K_L$. Associated with each $v_n$, Lemmas \ref{lem:Wv-representation}-- \ref{lem:stationary-profile} yield the corresponding functions $W_n$, $U_n$, $\eta_n$ and $\widetilde V_n$. By Lemmas \ref{lem:Wv-representation} and \ref{lem:U-basic}, for any compact interval $[-k,k]\subset\mathbb R$, there exists $C_k>0$, independent of $n$, such that
\begin{equation*}
\|W_n\|_{C^2([-k,k])}+\|U_n\|_{C^1([-k,k])}\le C_k.
\end{equation*}

Moreover, the lower barrier defining $\mathcal K_L$, together with the positivity of the Green kernel $G$, implies that $W_n$ has a positive lower bound on every compact interval, uniformly with respect to $n$. Combining this fact with the representation formula \eqref{eq:Uv-formula}, we obtain that, for every $k\in\mathbb N$, there exists $\delta_k\in(0,1)$, independent of $n$, such that
\begin{equation*}
0<U_n(z)\le 1-\delta_k,\qquad z\in[-k,k],\quad n\in\mathbb N.
\end{equation*}
Consequently, by hypothesis $(H_1)$, there exists $d_k>0$, independent of $n$, such that
\begin{equation*}
D(U_n(z))\ge d_k,\qquad z\in[-k,k],\quad n\in\mathbb N.
\end{equation*}
Hence, from the definition of $\eta_n$,
\begin{equation*}
\eta_n(z):=-\int_0^z\frac{ds}{D(U_n(s))},
\end{equation*}
the family $\{\eta_n\}$ is uniformly bounded and equicontinuous on $[-k,k]$. Thus there exists $M_k>0$, independent of $n$, such that
\begin{equation*}
\eta_n([-k,k])\subset[-M_k,M_k],\qquad n\in\mathbb N.
\end{equation*}

In addition, for $\eta\in\mathbb{R}$, Lemma \ref{lem:stationary-profile} gives the bound
\begin{equation*}
0<\widetilde V_n(\eta)<1.
\end{equation*}
Combining this bound with the interior Schauder estimates \cite[Chapter 6, Theorem 6.2 and Corollary 6.3]{Trudinger}, we obtain
\begin{equation*}
\|\widetilde V_n\|_{C^{2+\beta}([-M_k,M_k])}\le C_k,
\end{equation*}
for some $\beta\in(0,1)$, every $k\in\mathbb N$, and a constant $C_k>0$ independent of $n$. Using this estimate together with Lemma \ref{lem:stationary-profile}, we deduce that $F(v_n)$ is uniformly bounded in $\mathbb{R}$ and equicontinuous on every compact interval. By the Arzel\`a--Ascoli theorem and a diagonal argument, $\{F(v_n)\}$ has a locally uniformly convergent subsequence. Equivalently, this subsequence converges with respect to the norm $\|\cdot\|_*$. Therefore $F(\mathcal K_L)$ is relatively compact with respect to $\|\cdot\|_*$.

We continue to prove the continuity. Suppose $v_n\to v$ with respect to $\|\cdot\|_*$, or equivalently, $v_n\to v$ locally uniformly on $\mathbb R$, as $n\to+\infty$. Similarly, denote by $W_v$, $U_v$, $\eta_v$ and $\widetilde V_v$ the corresponding functions associated with $v$. Again from Lemmas \ref{lem:Wv-representation} and \ref{lem:U-basic}, we conclude that, as $n\to+\infty$,
\begin{equation*}
W_n\to W_v \quad\text{in }C^2_{\mathrm{loc}}(\mathbb R),\qquad U_n\to U_v \quad\text{in }C^1_{\mathrm{loc}}(\mathbb R).
\end{equation*}

It follows from the definition of $\eta_n$ and the local positivity of $D(U_v)$ that
\begin{equation*}
\eta_n\to\eta_v \quad\text{locally uniformly on }\mathbb R.
\end{equation*}
In particular, the inverse maps $Z_n$ converge locally uniformly to $Z_v$. Hence
\begin{equation*}
a_n(\eta):=rD(U_n(Z_n(\eta)))\to a_v(\eta):=rD(U_v(Z_v(\eta))) \quad\text{locally uniformly on }\mathbb R.
\end{equation*}
Let $b_n^\pm$, $H_n$, $K_n$, and $M_n$ be the quantities associated with $v_n$ in Lemma~\ref{lem:stationary-profile}, and set $\delta_n=b_n^+-b_n^-$. By the local convergence of $a_n$, the uniform bounds, and the uniqueness of $b_v^\pm$,
\begin{equation*}
b_n^\pm\longrightarrow b_v^\pm \qquad\text{in }C_{\mathrm{loc}}^1(\mathbb R).
\end{equation*}
The estimates \eqref{eq:U-right-uniform}, \eqref{eq:bpm-bound}, and \eqref{eq:delta-positive} also give
\begin{equation*}
\lim_{B\to+\infty}\sup_n\int_{-\infty}^{-B}\delta_n(s)\,ds=0.
\end{equation*}
Hence $M_n\to M_v$, and the corresponding super- and sub-solutions converge locally uniformly. By the local Schauder estimates, any subsequential limit solves the limiting stationary equation and lies between the limiting super- and sub-solutions; it therefore has the required end states and satisfies \eqref{eq:left-tail-eta}. The uniqueness in Lemma~\ref{lem:stationary-profile} gives
\begin{equation*}
\widetilde V_n\longrightarrow\widetilde V_v \qquad\text{in }C_{\mathrm{loc}}^1(\mathbb R).
\end{equation*}
Combining the locally uniform convergence of $\eta_n$ and the local $C^1$-convergence of $\widetilde V_n$, we deduce that
\begin{equation*}
F(v_n)(z)=\widetilde V_n(\eta_n(z))\to \widetilde V_v(\eta_v(z))=F(v)(z)
\end{equation*}
locally uniformly on $\mathbb R$. By the equivalence between locally uniform convergence and convergence with respect to $\|\cdot\|_*$ on $\mathcal K_L$, this is equivalent to
\begin{equation*}
\|F(v_n)-F(v)\|_*\to0.
\end{equation*}
Hence $F$ is continuous with respect to $\|\cdot\|_*$.

It is straightforward to verify that $\mathcal K_L$ is nonempty, closed and convex with respect to $\|\cdot\|_*$. Together with the invariance property $F(\mathcal K_L)\subset\mathcal K_L$ from Lemma \ref{lem:F-invariant}, we apply the Schauder--Tychonoff fixed-point theorem \cite[Theorem 1.13 p. 148]{Dugundji-Granas-1982} in the $\|\cdot\|_*$-topology to obtain a fixed point $V\in\mathcal K_L$ such that
\begin{equation*}
F(V)=V.
\end{equation*}
This completes the proof.
\end{proof}

\begin{proof}[Proof of Theorems~\ref{thm:main} and~\ref{thm:pro}]
By Lemma~\ref{lem:F-compact-continuous}, there exists $V\in\mathcal K_L$ such that
\begin{equation*}
F(V)=V.
\end{equation*}
Define
\begin{equation*}
W:=T_1(V),\qquad U:=T_2(W).
\end{equation*}
The construction of the three auxiliary maps implies that the triple $(U,V,W)$ satisfies \eqref{eq:tw-system}. Moreover, Lemmas~\ref{lem:Wv-representation}, \ref{lem:U-basic} and \ref{lem:stationary-profile} give
\begin{equation*}
0<U(z)<1,\qquad 0<V(z)<1,\qquad 0<W(z)<1, \qquad z\in\mathbb R.
\end{equation*}

We next prove that $V$ is strictly decreasing. Rewriting the second equation of \eqref{eq:tw-system} in divergence form, we have
\begin{equation*}
(AV')'=-\frac{rA}{D(U)}V(1-V)<0,
\end{equation*}
where
\begin{equation*}
A(z):=D(U(z))\exp\left\{\int_{z_0}^{z}\frac{\theta}{D(U(s))}\,ds\right\}>0
\end{equation*}
for some fixed $z_0\in\mathbb R$. Hence $AV'$ is strictly decreasing. We claim that $AV'<0$ on $\mathbb R$. If there were $z_1\in\mathbb R$ such that
\begin{equation*}
A(z_1)V'(z_1)\ge0,
\end{equation*}
then, for every $z<z_1$,
\begin{equation*}
A(z)V'(z)>A(z_1)V'(z_1)\ge0.
\end{equation*}
Since $A(z)>0$, this gives $V'(z)>0$ for $z<z_1$, and therefore
\begin{equation*}
V(z_1)\ge \lim_{z\to-\infty}V(z)=1,
\end{equation*}
which contradicts $0<V(z_1)<1$. Thus $A(z)V'(z)<0$ for all $z\in\mathbb R$, and consequently
\begin{equation*}
V'(z)<0,\qquad z\in\mathbb R.
\end{equation*}

The monotonicity of $W$ follows from the Green representation. Since
\begin{equation*}
W(z)=\int_{\mathbb R}G(y)V(z-y)\,dy,
\end{equation*}
with $G>0$, differentiating under the integral sign gives
\begin{equation*}
W'(z)=\int_{\mathbb R}G(y)V'(z-y)\,dy<0, \qquad z\in\mathbb R.
\end{equation*}

It remains to prove the monotonicity of $U$. Let
\begin{equation*}
\varphi(z):=1-d_1W(z).
\end{equation*}
Since $W'<0$, we have $\varphi'>0$. The first equation of \eqref{eq:tw-system} can be written as
\begin{equation}\label{eq:U-varphi-equation}
\theta U'=U(U-\varphi).
\end{equation}
We first show that
\begin{equation}\label{eq:U-greater-varphi}
U(z)\ge\varphi(z),\qquad z\in\mathbb R.
\end{equation}
Indeed,
\begin{equation*}
U(-\infty)=0,\qquad \varphi(-\infty)=1-d_1<0,
\end{equation*}
while
\begin{equation*}
U(+\infty)=1,\qquad \varphi(+\infty)=1.
\end{equation*}
If \eqref{eq:U-greater-varphi} failed, then there would exist a nonempty interval on which $U<\varphi$. On such an interval, \eqref{eq:U-varphi-equation} gives $U'<0$, whereas $\varphi'>0$. Hence $U-\varphi$ is strictly decreasing there, and once it becomes negative it cannot return to zero in the forward direction. This contradicts
\begin{equation*}
\lim_{z\to+\infty}(U(z)-\varphi(z))=0.
\end{equation*}
Therefore \eqref{eq:U-greater-varphi} holds. If $U(z_0)=\varphi(z_0)$ at some finite $z_0$, then \eqref{eq:U-varphi-equation} gives $U'(z_0)=0$, and hence $(U-\varphi)'(z_0)=-\varphi'(z_0)<0$, contradicting $U-\varphi\ge0$. Thus $U>\varphi$ on $\mathbb R$. In particular, $1-U<d_1W$, and \eqref{eq:U-varphi-equation} gives
\begin{equation*}
U'(z)>0,\qquad z\in\mathbb R.
\end{equation*}
Finally, since $V\in\mathcal K_L$, the definition of $\mathcal K_L$ and \eqref{eq:barrier-limits-L} imply
\begin{equation*}
V(-\infty)=1,\qquad V(+\infty)=0.
\end{equation*}
By Lemmas~\ref{lem:Wv-representation} and~\ref{lem:U-basic}, we have
\begin{equation*}
(U,V,W)(-\infty)=(0,1,1),\qquad (U,V,W)(+\infty)=(1,0,0).
\end{equation*}
This proves Theorems~\ref{thm:main} and~\ref{thm:pro}.
\end{proof}

\begin{proposition}[Asymptotic behavior]\label{prop:asymptotic-estimates}
Let $(U,V,W)$ be the traveling wave constructed in Theorem~\ref{thm:main}, and let $\eta$ be the degenerate variable defined in \eqref{defeta}. Then there exists a constant $C_1>0$ such that
\begin{equation}\label{eq:main-V-right-new}
0<V(z)\le C_1e^{-\frac{rz}{\theta}}, \qquad z\ge0,
\end{equation}
and, for every $0<\rho<\nu_*$, there exists a constant $C_2=C_2(\rho)>0$ such that
\begin{equation}\label{eq:main-V-left-new}
0<1-V(z)\le C_2e^{-\rho\eta(z)}, \qquad z\le0.
\end{equation}
Consequently, for every $0<\sigma_+<\min\left\{\frac{r}{\theta},-\lambda_2^-\right\}$, there exists $C_+>0$ such that
\begin{equation}\label{eq:main-right-tail-new}
0<V(z)\le C_+e^{-\sigma_+z},\quad 0<W(z)\le C_+e^{-\sigma_+z},\quad 0<1-U(z)\le C_+e^{-\sigma_+z},\quad z\ge0.
\end{equation}
Moreover, for every $0<\rho<\nu_*$ and every $0<\sigma_-<\min\{ \frac{\rho}{D(0)},\lambda_2^+\}$, there exists $C_->0$ such that
\begin{equation}\label{eq:main-left-tail-new}
0<1-V(z)\le C_-e^{\sigma_-z},\qquad 0<1-W(z)\le C_-e^{\sigma_-z},\qquad z\le0.
\end{equation}
In addition, there exist $C_U>0$ and $\sigma_U>0$ such that
\begin{equation}\label{eq:main-U-left-new}
0<U(z)\le C_Ue^{\sigma_Uz},\qquad z\le0.
\end{equation}
Finally, $U'(z)\to0$ and $W'(z)\to0$ as $z\to\pm\infty$. If $V'$ is uniformly continuous on $\mathbb R$, then $V'(z)\to0$ as $z\to\pm\infty$.
\end{proposition}

\begin{proof}
Since the fixed point satisfies $V(z)=\widetilde V(\eta(z))$, \eqref{eq:main-V-right-new} follows from \eqref{eq:F-right-uniform-new}, while \eqref{eq:main-V-left-new} follows from \eqref{eq:right-tail-eta}. We next derive the remaining estimates in the original variable $z$.

For $z\ge0$, \eqref{eq:main-V-right-new} implies $V(z)\le Ce^{-\sigma_+z}$ for every $0<\sigma_+<\frac{r}{\theta}$. Combining this with the Green representation
\begin{equation*}
W(z)=\int_{\mathbb R}G(y)V(z-y)\,dy
\end{equation*}
and the exponential decay of $G$, we obtain $W(z)\le Ce^{-\sigma_+z}$, provided $\sigma_+<-\lambda_2^-$. Since the monotonicity proof gives $1-U<d_1W$, the right-hand estimates in \eqref{eq:main-right-tail-new} follow.

For $z\le0$, \eqref{eq:eta-growth} gives $\eta(z)\ge -\frac{z}{D(0)}$. Thus \eqref{eq:main-V-left-new} yields $1-V(z)\le Ce^{\sigma_-z}$ for every $0<\sigma_-<\frac{\rho}{D(0)}$. Using $\int_{\mathbb R}G(y)\,dy=1$, we write
\begin{equation*}
1-W(z)=\int_{\mathbb R}G(y)(1-V(z-y))\,dy.
\end{equation*}
The exponential decay of $G$ then gives the same left-hand estimate for $1-W$, provided $\sigma_-<\lambda_2^+$. This proves \eqref{eq:main-left-tail-new}.

It remains to estimate $U$ at the invaded end. Since $W(-\infty)=1$ and $d_1>1$, we can choose $z_0<0$ and $\delta_U>0$ such that $d_1W(z)-1\ge\delta_U$ for $z\le z_0$. From the first traveling-wave equation, $\theta U'=U(U+d_1W-1)$, and hence $(\ln U)'\ge\frac{\delta_U}{\theta}$ for $z\le z_0$. Integrating over $(z,z_0)$ gives \eqref{eq:main-U-left-new} after enlarging the constant.

The derivative limits follow from the same estimates. Indeed, the identity $\theta U'=U(U+d_1W-1)$ and the limits of $U$ and $W$ imply $U'(z)\to0$ as $z\to\pm\infty$. The formula $W'(z)=\lambda_2^-I_1(z)+\lambda_2^+I_2(z)$, obtained from the Green representation, together with the right-tail estimate for $V$ and the left-tail estimate for $1-V$, gives $W'(z)\to0$ as $z\to\pm\infty$. Finally, if $V'$ is uniformly continuous, then $V'<0$ and $\int_{\mathbb R}|V'(z)|\,dz=V(-\infty)-V(+\infty)=1$; hence $V'(z)\to0$ at both infinities.
\end{proof}

\section*{Acknowledgments}
Xiong Li was supported by the National Natural Science Foundation of China (No. 12371158). Yang Wang was supported by the Shanxi Scholarship Council of China (No. 2025-070). Xinyue Cao was supported by the China Scholarship Council (No. 202506040020). The authors would like to thank the anonymous referees for their valuable comments and suggestions.

\bibliographystyle{elsarticle-num}

\end{document}